\documentclass[reqno]{amsart}
\usepackage{graphicx}
\usepackage{pictex}
\usepackage[all]{xy}
\usepackage{amscd}
\usepackage{amssymb}
\usepackage{amsmath}
\usepackage{amsthm}
\usepackage{amsfonts}
\usepackage{txfonts}
\usepackage{graphics}
\usepackage{epsfig,psfrag}
\usepackage[english]{babel}
\usepackage{latexsym}
\usepackage{mathrsfs}
\usepackage{color}

\newtheorem{theorem}{Theorem}[section]

\newtheorem{lemma}[theorem]{Lemma}

\newtheorem{proposition}[theorem]{Proposition}

\theoremstyle{definition}

\newtheorem{definition}[theorem]{Definition}

\begin{document}
\title{Length minima for an infinite family of filling closed curves on a one-holed torus} 

\author{Zhongzi Wang}
\address{School of Mathematical Sciences, Peking University, Beijing 100871, China}
\email{wangzz22@stu.pku.edu.cn}
\author{Ying Zhang}
\address{School of Mathematical Sciences, Soochow University, Suzhou 215006, China}
\email{yzhang@suda.edu.cn}

\begin{abstract}
We explicitly find the minima as well as the minimum points of the geodesic length functions for the family of filling (hence non-simple) closed curves, $a^2b^n$ ($n\ge 3$), on a complete one-holed hyperbolic torus in its relative Teichm\"uller space, where  $a, b$ are simple closed curves on the one-holed torus which intersect exactly once transversely. This provides concrete examples for the problem to minimize the geodesic length of a fixed filling closed curve on a complete hyperbolic surface of finite type in its relative Teichm\"uller space.  
\end{abstract}

\date{}
\maketitle


\section{Introduction}

The length of any non-simple closed geodesic on an orientable, complete hyperbolic surface is known to have a positive lower bound (see Yamada \cite{yamada1982}, also Hempel \cite{hempel1984}). In \cite{basmajian1993} and \cite{basmajian2013} Basmajian obtained lower bounds for the length of a non-simple closed geodesic in terms of its self-intersection numbers. 
In \cite{Shen-Wang} Shen--Wang optimized the asymptotic multiplier in Basmajian's lower bound. In a slightly different direction, 
Erlandsson--Parlier \cite{erlandsson-parlier2020} initiated the study of the self-intersection number of the shortest non-simple closed geodesics with at least $k$ self-intersections. Advancing further from \cite{vo2022}, Basmajian--Parlier--Vo \cite{basmajian-parlier-vo2022arxiv} showed that, for $k\ge 10^{125}$, the shortest among all non-simple closed geodesics with at least $k$ self-intersections on any orientable, complete hyperbolic surface $S$ is uniquely realized by a corkscrew geodesic with exactly $k$ self-intersections on an ideal pair of pants (a corkscrew geodesics is that represented by word $ab^{k}$ in the fundamental group of the pair of pants where $a$ and $b$ are simple closed curves each surrounding a different cusp).

Let $S$ be an orientable surface of finite type equipped with a complete hyperbolic structure. 
Then $S$ has a finite number (possible none) of holes and cusps. 
Surrounding each hole of $S$ there is a unique simple closed geodesic which we call the neck geodesic of the hole. 
Though there is no simple closed geodesics surrounding a cusp, we shall still say that a cusp has a null neck geodesic of length zero at infinity. 

It is a well-known fact that a homotopically non-trivial, non-peripheral closed curve $c$ on a complete hyperbolic surface $S$ can be realized as a unique closed geodesic which might not be primitive (see, for example, Buser \cite[Theorem 1.6.6]{buser1992}); we denote its geodesic length by $L_c$.
By definition, the relative Teichm\"uller space of $S$ consists of all the marked complete hyperbolic structures on the same topological surface with all the lengths of the neck geodesics fixed.

We address the problem to minimize the geodesic length of a fixed filling (hence non-simple) closed curve on a finite type surface in the relative Teichm\"uller space of the surface. It will be shown in a future paper that such a geodesic length function has exactly one minimum point in the relative Teichm\"uller space. 

In this paper, we solve the problem for a concrete infinite family of filling closed curves, $a^2b^n$ ($n \ge 3$), on a complete one-holed hyperbolic torus $S_{1,1}$, where $a, b$ are simple closed curves on $S_{1,1}$ which intersect once transversely. 
Precisely, we find the minima and the corresponding minimum points for the geodesic lengths of the filling closed curves, $a^2b^n$ ($n \ge 3$), on $S_{1,1}$ in its relative Teichm\"uller space (see Theorems \ref{thm:a2bn} and \ref{rmk:a2bn}). 

Note that the closed curve $a^2b^n$ ($n \ge 3$) on $S_{1,1}$ has self-intersection number $n-1$ (see \cite[Proposition 6.1]{chas-phillips2010em}). Figure \ref{fig:a2b4} shows $a^2b^4$ drawn on a cut-open torus. 

\begin{figure}
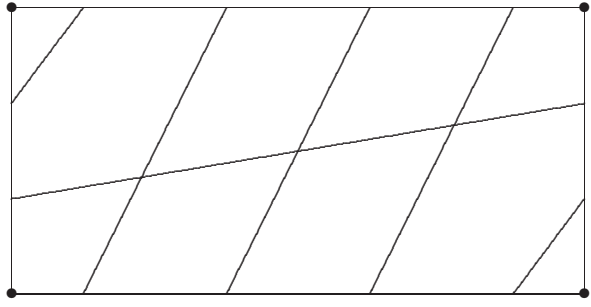

\begin{center}
\mbox{\beginpicture \setcoordinatesystem units <0.025in,0.025in>
\setplotarea x from -70 to 70 , y from 0 to 65
\plot -60 20 60 40 / 
\plot -60 40 -45 60 / 
\plot -45 0 -15 60 /
\plot -15 0 15 60 /
\plot 15 0 45 60 /
\plot 45 0 60 20 /
\linethickness=0.10pt
\putrule from -60 0.1 to 60 0.1
\putrule from 60 0 to 60 60
\putrule from -60 60.2 to 60 60.2
\putrule from -60 0 to -60 60
\put {\mbox{$\bullet$}} [cc] <0mm,0mm> at -60 0  %
\put {\mbox{$\bullet$}} [cc] <0mm,0mm> at -60 60  %
\put {\mbox{$\bullet$}} [cc] <0mm,0mm> at 60 0  %
\put {\mbox{$\bullet$}} [cc] <0mm,0mm> at 60 60  %
\endpicture}
\end{center}
\caption{Curve $a^2b^4$ drawn on a cut-open torus (with bullet for the hole)} \label{fig:a2b4}
\end{figure}

The rest of the paper is organized as follows. In \S 2, for $X,Y\in {\rm SL}(2,\mathbb C)$ and $m\in\mathbb Z$, we derive nice expressions of ${\rm tr}(XY^{2m})$ and ${\rm tr}(XY^{2m+1})$ (see Proposition \ref{prop:trace-of-a2bn}), which implies symmetry of minimum points for the geodesic length function $L_{a^2b^n}$. In \S 3 we present some basic facts on the geometry of a complete one-holed hyperbolic torus and give a parametrization of the relative Teichm\"uller space. 
Finally, in \S 4 we explicitly find (see Theorem \ref{thm:a2bn}) the minima and the minimum points of the geodesic length functions $L_{a^2b^n}$ ($n \ge 3$) in the relative Teichm\"uller space of a complete one-holed hyperbolic torus $S_{1,1}$ (where $a, b$ are simple closed curves on  $S_{1,1}$ which intersect once transversely), and, furthermore, determine the asymptotic behaviours of the length minima as one of $n$ and the length of the neck geodesic tends to infinity while the other remains fixed (see Theorem \ref{rmk:a2bn}).


\section{Trace of $X^2Y^{n}$ for matrices $X$ and $Y$ in ${\rm SL}(2,\mathbb C)$}\label{s:calculation}
Let $X,Y \in {\rm SL}(2,\mathbb C)$ and $m\in\mathbb Z$. 
We shall derive (see Proposition \ref{prop:trace-of-a2bn}) expressions for ${\rm tr}\,(X^2Y^{2m})$ and ${\rm tr}\,(X^2Y^{2m+1})$ in terms of ${\rm tr}(X)$, ${\rm tr}(Y)$, ${\rm tr}(XY)$, ${\rm tr}(XY^{m})$ and ${\rm tr}(XY^{m\pm 1})$ which supply symmetry of minimum points of the geodesic length function $L_{a^2b^{n}}$. 

The following well-known Fricke--Klein trace identity for ${\rm SL}(2,\mathbb C)$-matrices will be used throughout this section: 
\begin{equation}\label{eqn:traceidFK}
{\rm tr}(XY) + {\rm tr}(XY^{-1}) = {\rm tr}(X)\,{\rm tr}(Y).
\end{equation}
It can be proved by induction that, for any word $w$ in letters $X, Y \in {\rm SL}(2,\mathbb C)$, its trace ${\rm tr}(w)$ is polynomial in ${\rm tr}(X)$,  ${\rm tr}(Y)$ and ${\rm tr}(XY)$. As an example, we have
\begin{equation}\label{eqn:trace-of-comm}
{\rm tr}(XYX^{-1}Y^{-1}) = {\rm tr}^2(X)+{\rm tr}^2(Y)+{\rm tr}^2(XY)-{\rm tr}(X)\,{\rm tr}(Y)\,{\rm tr}(XY)-2. 
\end{equation}
Proofs of the identities and facts just mentioned can be found in \cite[pp.619--621]{goldman2009irma}. 

\begin{lemma}\label{lem:traceidhe} 
For any word $w$ in letters $X,Y \in {\rm SL}(2,\mathbb C)$, we have
\begin{eqnarray}\label{eqn:hyperelliptic}
{\rm tr}(w(X,Y))={\rm tr}(w(X^{-1},Y^{-1})).
\end{eqnarray}
In particular, for $m \in \mathbb Z$, 
\begin{eqnarray}
{\rm tr}(XY^{m}X^{-1}Y^{-m-1})={\rm tr}(X^{-1}Y^{-m}XY^{m+1}).
\end{eqnarray}
\end{lemma}

\begin{proof}
We prove identity (\ref{eqn:hyperelliptic}) by induction on the word length of $w$. It is easy to verify that (\ref{eqn:hyperelliptic}) is true if the length of $w$ is at most four. Now suppose $w$ is of length at least $5$. We assume  by induction that (\ref{eqn:hyperelliptic}) is true for all words of length less than that of $w$. 

By conjugating $w$ by a subword, we may assume the first letter $Z \in \{X,Y,X^{-1},Y^{-1}\}$ of $w$ is the same as a later letter in $w$. Then $w=Zw_1Zw_2$, where $w_1, w_2$ are subwords of $w$ (one of $w_1$ and $w_2$ might be empty), and we have 
\begin{equation}
w(X^{-1},Y^{-1})=Z^{-1}w_1(X^{-1},Y^{-1})\,Z^{-1}w_2(X^{-1},Y^{-1}). 
\end{equation}
By Fricke--Klein trace identity (\ref{eqn:traceidFK}), we have
\begin{eqnarray*}
& & \hspace{-12pt}  {\rm tr}(w(X^{-1},Y^{-1}))  \\
& \!  =    & {\rm tr}\big(Z^{-1}w_1(X^{-1},Y^{-1})\big){\rm tr}\big(Z^{-1}w_2(X^{-1},Y^{-1})\big)-{\rm tr}\big(w_1(X^{-1},Y^{-1})\,w_2(X^{-1},Y^{-1})^{-1}\big) \\
& \!  =    & {\rm tr}\big(Zw_1(X,Y)\big)\,{\rm tr}\big(Zw_2(X,Y)\big)-{\rm tr}\big(w_1(X,Y)\,w_2(X,Y)^{-1}\big) \\
& \!  =    & {\rm tr}\big(Zw_1(X,Y)\big)\,{\rm tr}\big(Zw_2(X,Y)\big)-{\rm tr}\big(Zw_1(X,Y)\,(Zw_2(X,Y))^{-1}\big) \\
& \!  =    & {\rm tr}(Zw_1(X,Y)\,Zw_2(X,Y)),
\end{eqnarray*}
which is (\ref{eqn:hyperelliptic}), as desired. This proves Lemma \ref{lem:traceidhe}.
\end{proof}

We shall always write
\begin{eqnarray}
x={\rm tr}(X), \quad y={\rm tr}(Y), \quad z={\rm tr}(XY), \quad \mu={\rm tr}(XYX^{-1}Y^{-1})+2.
\end{eqnarray}
Then (\ref{eqn:trace-of-comm}) becomes
\begin{equation}
x^2+y^2+z^2-xyz=\mu,
\end{equation}
which can be rewritten as either of the following two identities:
\begin{eqnarray}
(xy-2z)^2 &=& x^2(y^2-4)-4(y^2-\mu), \label{eqn:(xy-z)^2} \\
xz(y-2) &=& (x-z)^2 + y^2-\mu. \label{eqn:xz(y-2)}
\end{eqnarray}

Let us write, for all $m \in \mathbb Z$, $x_m={\rm tr}(XY^{m})$. 
By (\ref{eqn:traceidFK}), we have
\begin{eqnarray} 
x_{m-1}+x_{m+1}=x_{m}y, 
\end{eqnarray}
and hence $x_{m}y-2x_{m+1}=x_{m-1}-x_{m+1}$. 

Note that $(XY^{m})Y(XY^{m})^{-1}Y^{-1}=XYX^{-1}Y^{-1}$. 
Applying (\ref{eqn:trace-of-comm}) gives
\begin{eqnarray}\label{eqn:xmyxm+1}
x_{m}^2+y^2+x_{m+1}^2-x_{m}yx_{m+1} = \mu.
\end{eqnarray}
By (\ref{eqn:(xy-z)^2}), we have  $(x_{m}y-2x_{m+1})^2=x_{m}^2(y^2-4)-4(y^2-\mu)$, and hence
\begin{eqnarray}\label{eqn:x_m^2}
x_{m}^2=\frac{(x_{m-1}-x_{m+1})^2+4(y^2-\mu)}{y^2-4}.
\end{eqnarray}
By (\ref{eqn:xz(y-2)}), we have $x_{m}x_{m+1}(y-2) = (x_{m}-x_{m+1})^2 + y^2-\mu$,
and hence
\begin{equation} \label{eqn:xmxm+1}
x_{m}x_{m+1} = \frac{(x_{m}-x_{m+1})^2 + y^2-\mu}{y-2}.
\end{equation}

Combining $x_{m-2}+x_{m}=yx_{m-1}$ and $x_{m-1}^2+y^2+x_{m}^2-x_{m-1}yx_{m} = \mu$ gives 
\begin{equation} 
x_{m-2}x_{m} - x_{m-1}^2 = y^2-\mu.
\end{equation}

\begin{definition}\label{def:P_n}
For each $n \in \mathbb Z$, we define a polynomial $P_n$ in $x,y,z$ as follows:

1) If $n=2m$ is even, define
\begin{eqnarray}
P_{2m}=-\,{\rm tr}(XY^{m}X^{-1}Y^{-m});
\end{eqnarray}

2) If $n=2m+1$ is odd, define
\begin{eqnarray}
P_{2m+1}=-\,{\rm tr}(XY^{m}X^{-1}Y^{-m-1})=-\,{\rm tr}(XY^{m+1}X^{-1}Y^{-m}).
\end{eqnarray}
In particular, we have $P_{0}=-2$ and $P_{1}=-y$.
\end{definition}

The polynomials $P_n$ can be characterized by recursive relations as follows.

\begin{proposition}\label{prop:recursive}
The polynomials $P_n$,  $n \in \mathbb Z$ in $x,y,z$ satisfy recursive relations
\begin{eqnarray}
P_{2m}  &=&  yP_{2m-1}-P_{2m-2}+y^2-\mu, \label{eqn:recursive-even}\\
 P_{2m+1}  &=&  yP_{2m}-P_{2m-1}, \label{eqn:recursive-odd}
\end{eqnarray}
with $P_{0}=-2$ and $P_{1}=-y$.
\end{proposition}

\begin{proof}
It is easy to see that 
\begin{eqnarray}
&  & P_{0}=-\,{\rm tr}(XY^{0}X^{-1}Y^{0})=-2,  \\
&  & P_{1}=-\,{\rm tr}(XY^{0}X^{-1}Y^{-1})=-y.
\end{eqnarray}
We first prove (\ref{eqn:recursive-even}). As
\begin{eqnarray}
P_{2m}  &=&  -\,{\rm tr}(XY^{m}X^{-1}Y^{-m+1}Y^{-1})  \nonumber  \\
  &=&  -\,{\rm tr}(XY^{m}X^{-1}Y^{-m+1})\,{\rm tr}(Y^{-1}) + {\rm tr}(XY^{m}X^{-1}Y^{-m+2}) \nonumber  \\
 &=&  y P_{2m-1}  +  {\rm tr}(XY^{m}X^{-1}Y^{-m+2}),
\end{eqnarray}
it remains to show that
\begin{eqnarray}\label{eqn:remains-to-show}
{\rm tr}(XY^{m}X^{-1}Y^{-m+2}) - {\rm tr}(XY^{m-1}X^{-1}Y^{-m+1}) = y^2-\mu. 
\end{eqnarray}
By Fricke--Klein trace identity (\ref{eqn:traceidFK}), we have
\begin{eqnarray}
{\rm tr}(XY^{m}X^{-1}Y^{-m+2}) &   =   & {\rm tr}(XY^{m})\,{\rm tr}(X^{-1}Y^{-m+2})- {\rm tr}(XY^{2m-2}X),  \\
{\rm tr}(XY^{m-1}X^{-1}Y^{-m+1})  &   =   & {\rm tr}(XY^{m-1})\,{\rm tr}(X^{-1}Y^{-m+1})- {\rm tr}(XY^{2m-2}X). 
\end{eqnarray}
Hence
\begin{eqnarray}
{\rm tr}(XY^{m}X^{-1}Y^{-m+2}) - {\rm tr}(XY^{m-1}X^{-1}Y^{-m+1})  = x_{m}x_{m-2}-x_{m-1}^2 = y^2 - \mu.
\end{eqnarray}
This proves (\ref{eqn:recursive-even}). For (\ref{eqn:recursive-odd}), we have
\begin{eqnarray}
P_{2m+1}  &=&  -\,{\rm tr}(XY^{m}X^{-1}Y^{-m}Y^{-1}) \nonumber \\
  &=&  -\,{\rm tr}(XY^{m}X^{-1}Y^{-m})\,{\rm tr}(Y^{-1}) + {\rm tr}(XY^{m}X^{-1}Y^{-m+1}) \nonumber \\
 &=&  y P_{2m}  -  P_{2m-1}.
\end{eqnarray}
This finishes the proof of Proposition \ref{prop:recursive}.
\end{proof}

It follows from Proposition \ref{prop:recursive} that $P_n$ can be rewritten as a polynomial in $y$ and $\mu$.

\begin{proposition}\label{prop:P-P}
There exist univariate polynomials $q_n$, $n \ge 2$ 
such that
\begin{eqnarray}
P_{2m}-P_{0}  &=&  (4-\mu)\,q_{2m}(y-2),  \label{eqn:P-P 2m} \\
P_{2m+1}-P_{1}  &=&  (4-\mu)\,q_{2m+1}(y-2). \label{eqn:P-P 2m+1} 
\end{eqnarray}
Precisely, $q_2(t)=1$, $q_3(t)=t+2$, and $q_{n}$ satisfies the following recursive relations:
\begin{eqnarray}\label{eqn:Qrecursive}
q_{2m}(t)  &=&  (t+2)\,q_{2m-1}(t)-q_{2m-2}(t)+1, \\
q_{2m+1}(t)  &=&  (t+2)\,q_{2m}(t)-q_{2m-1}(t).
\end{eqnarray}
 Furthermore, the coefficients of $q_n$ for $n \ge 2$ are all positive integers.
\end{proposition}

\begin{proof} Applying Proposition \ref{prop:recursive} and by induction, we have
\begin{eqnarray*}
P_{2m} - P_{0} 
 &=&   y(P_{2m-1}-P_{1})-(P_{2m-2}- P_{0})+4-\mu \\
 &=&   y(4-\mu)q_{2m-1}(y-2)-(4-\mu)q_{2m-2}(y-2)+4-\mu \\
 &=&   (4-\mu) \big[ yq_{2m-1}(y-2)-q_{2m-2}(y-2)+1 \big], 
\end{eqnarray*}
and 
\begin{eqnarray*}
P_{2m+1} - P_{1} 
 &=&   y(P_{2m}-P_{0})-(P_{2m-1}- P_{1}) \\
 &=&   y(4-\mu)q_{2m}(y-2)-(4-\mu)q_{2m-1}(y-2) \\
 &=&   (4-\mu) \big[ yq_{2m}(y-2)-q_{2m-1}(y-2) \big], 
\end{eqnarray*}
This proves all the formulas in Proposition \ref{prop:P-P}.

It remains to prove that the coefficients of $q_n$ for $n \ge 2$ are positive integers.
Actually, it is easily shown by induction that, for $n \ge 2$, the coefficients of $q_n-q_{n-1}$ are all positive integers. 
This finishes the proof of Proposition \ref{prop:P-P}.
\end{proof}

Now we are in a position to obtain expressions for ${\rm tr}(X^2Y^{n})$, $n\ge 3$, 
which will be used in \S \ref{s:minima} to minimize the geodesic length of $a^2b^n$.

\begin{proposition}\label{prop:trace-of-a2bn}
For $m \in \mathbb Z$, we have
\begin{eqnarray}
{\rm tr}(X^2Y^{2m})   &=&     \frac{(x_{m-1}-x_{m+1})^2}{y^2-4}+2+(4-\mu)\big[ 4(y^2-4)^{-1}+q_{2m}(y-2)],   
\label{eqn:a2b2m} \\
{\rm tr}(X^2Y^{2m+1})   &=&    \frac{(x_m-x_{m+1})^2}{y-2}+2+(4-\mu)\big[ (y-2)^{-1}+q_{2m+1}(y-2)]. 
\label{eqn:a2b2m+1}
\end{eqnarray}
\end{proposition}

\begin{proof} We apply (\ref{eqn:traceidFK}), (\ref{eqn:x_m^2}), (\ref{eqn:xmxm+1}), (\ref{eqn:P-P 2m}) and (\ref{eqn:P-P 2m+1}) 
to obtain (\ref{eqn:a2b2m}) and (\ref{eqn:a2b2m+1}) as follows:
\begin{eqnarray*}
{\rm tr}(X^2Y^{2m}) 
  &=&    {\rm tr}(XY^{m}Y^{m}X)   \\
  &=&    {\rm tr}(XY^{m})\,{\rm tr}(Y^{m}X) - {\rm tr}(XY^{m}X^{-1}Y^{-m}) \\
  &=&    x_{m}^2 + P_{2m} \\
  &=&     \frac{(x_{m-1}-x_{m+1})^2}{y^2-4}+2+(4-\mu)\big[ 4(y^2-4)^{-1}+q_{2m}(y-2)], \\
{\rm tr}(X^2Y^{2m+1}) 
  &=&    {\rm tr}(XY^{m}Y^{m+1}X)   \\
  &=&    {\rm tr}(XY^{m})\,{\rm tr}(Y^{m+1}X) - {\rm tr}(XY^{m}X^{-1}Y^{-m-1}) \\
  &=&    x_{m}x_{m+1} + P_{2m+1} \\
  &=&    \frac{(x_m-x_{m+1})^2}{y-2}+2+(4-\mu)\big[ (y-2)^{-1}+q_{2m+1}(y-2)]. 
\end{eqnarray*}
This proves Proposition \ref{prop:trace-of-a2bn}. 
\end{proof}



\section{Geometry of a complete one-holed hyperbolic torus} 
Let $S_{1,1}$ be a one-holed torus, and let $a,b$ be two simple closed curves on $S_{1,1}$ which intersect once transversely.
The commutator $aba^{-1}b^{-1}$ then represents the simple closed curve $\partial$ surrounding the hole. 

We shall use the upper-half complex plane model $H^2$ of the hyperbolic plane. 
An orientation-preserving isometry of  $H^2$ is then given by a fractional linear transformation mapping the upper-half complex plane onto itself. 
As a consequence, the group of orientation-preserving isometries of  $H^2$ coincides with ${\rm PSL}(2, \mathbb R)$.

A complete hyperbolic structure on $S_{1,1}$ gives rise to a holonomy representation 
\begin{equation}
\eta: \pi_1(S_{1,1}) \rightarrow {\rm PSL}(2, \mathbb R), 
\end{equation}
which is unique up to simultaneous conjugation by elements in ${\rm PSL}(2, \mathbb R)$. 

As the axes of $\eta(a)$ and $\eta(b)$ cross each other, a lifted representation
\begin{equation}
\rho: \pi_1(S_{1,1}) \rightarrow {\rm SL}(2, \mathbb R)
\end{equation}
of $\eta$ can be chosen so that ${\rm tr}\,\rho(a)>2$, ${\rm tr}\,\rho(b)>2$, ${\rm tr}\,\rho(ab)>2$.  
We write
\begin{equation}
x={\rm tr}\,\rho(a), \quad y={\rm tr}\,\rho(b), \quad z={\rm tr}\,\rho(ab).
\end{equation}
 Then
\begin{equation}
x=2\cosh\frac{L_{a}}{2}, \quad y=2\cosh\frac{L_{b}}{2}, \quad z=2\cosh\frac{L_{ab}}{2}.
\end{equation}

Using the fact that the axes of $\eta(a)$ and $\eta(b)$ cross each other, direct calculations show that
${\rm tr}\,\rho(aba^{-1}b^{-1}) <2$, and hence ${\rm tr}\,\rho (aba^{-1}b^{-1}) \le -2$.
It follows that 
\begin{equation}\label{eqn:necktraceisneg}
{\rm tr}\,\rho (aba^{-1}b^{-1}) = -2\cosh\frac{L_{\partial}}{2}. 
\end{equation}

On the other hand, by (\ref{eqn:trace-of-comm}), we have ${\rm tr}\,\rho (aba^{-1}b^{-1})=x^2+y^2+z^2-xyz-2$. 
Write $\mu= 2+{\rm tr}\,\rho (aba^{-1}b^{-1})$. Then $\mu=2-2\cosh\frac{L_{\partial}}{2} \le 0$ and 
\begin{equation}
x^2+y^2+z^2-xyz = \mu.
\end{equation}

The relative Teichm\"{u}ller space ${\rm Teich}_{1,1}(L_{\partial})$ of all the marked complete hyperbolic structures on $S_{1,1}$ with fixed length $L_{\partial} \ge 0$ of neck geodesic is then parametrized as
\begin{eqnarray}
{\rm Teich}^{\rm length}_{1,1}(L_{\partial}) = \{ (L_{a},L_{b},L_{ab})\in (\mathbb R_{>0})^3 \mid x^2+y^2+z^2-xyz=\mu \},
\end{eqnarray}
where $x={\rm tr}\,\rho(a)$, $y={\rm tr}\,\rho(b)$, $z={\rm tr}\,\rho(ab)$ and $\mu= 2-2\cosh\frac{L_{\partial}}{2} \le 0$.

The simple closed geodesics representing the simple closed curves $a$, $b$ and $ab$ intersect pairwise at the so-called {\it Weierstrass points}  and thus form a hyperbolic triangle (called the {\it Weierstrass triangle}) of side-lengths $\frac{1}{2}L_{a}$, $\frac{1}{2}L_{b}$ and $\frac{1}{2}L_{ab}$.

The following proposition concerns the geometric nature of $L_{\partial}$. 

\begin{proposition}\label{propn:sineplus}
Let $S_{1,1}$ be a complete one-holed hyperbolic torus and let $a,b$ be simple closed geodesics on $S_{1,1}$ which intersect once transversely. Then
\begin{equation}\label{eqn:parea1} 
 \sinh\frac{L_a}{2}\sinh\frac{L_b}{2}\sin\theta = \cosh\frac{L_{\partial}}{4}, 
\end{equation}
where $\theta \in (0,\pi)$ is the interior angle of the associated Weierstrass triangle contained by the two sides corresponding to $a$ and $b$.
The constraint (\ref{eqn:parea1}) is equivalent to
\begin{equation}
 \sinh H_b \sinh\frac{L_b}{2} = \cosh\frac{L_{\partial}}{4}, \label{eqn:parea2}
\end{equation}
where $H_b$ is the length of the altitude perpendicular to the side of length $\frac{1}{2}L_{b}$.
\end{proposition}

\begin{proof}
By the Laws of Cosines for the Weierstrass triangle, we have
\begin{equation}
\cosh\frac{L_{ab}}{2}=\cosh\frac{L_a}{2}\cosh\frac{L_b}{2}-\sinh\frac{L_a}{2}\sinh\frac{L_b}{2}\cos\theta.
\end{equation}
The desired identity (\ref{eqn:parea1}) follows by direct calculations as follows:
\begin{eqnarray*}
& & \hspace{-12pt} \sinh\frac{L_a}{2}\sinh\frac{L_b}{2}\sin\theta \\
&\! =\! & \sqrt{\left(\sinh\frac{L_a}{2} \sinh\frac{L_b}{2}\right)^2
-\left(\cosh\frac{L_a}{2}\cosh\frac{L_b}{2}-\cosh\frac{L_{ab}}{2}\right)^2 } \\
&\! =\! & \sqrt{1+2\cosh\frac{L_a}{2}\cosh\frac{L_a}{2}\cosh\frac{L_{ab}}{2}
-\cosh^2\frac{L_a}{2}-\cosh^2\frac{L_b}{2}-\cosh^2\frac{L_{ab}}{2}} \\
&\! =\! & \sqrt{1-\frac{\mu}{4}} \\
&\! =\! & \cosh\frac{L_{\partial}}{4}.
\end{eqnarray*}
By the Law of Sines for a right-angled triangle, 
\begin{equation}
\sinh H_b=\sinh\frac{L_a}{2}\sin\theta.
\end{equation}
Hence the constraint (\ref{eqn:parea1}) can be replaced by (\ref{eqn:parea2}).
This proves Proposition \ref{propn:sineplus}.
\end{proof}


\section{Length minima of closed curves $a^2b^n$ on one-holed tori}\label{s:minima}

In this section we solve the problem of minimizing the geodesic length of the filling closed curve $a^2b^n$ ($n\ge 3$) in the relative Teichm\"uller space ${\rm Teich}_{1,1}(L_{\partial})$.

\begin{theorem}\label{thm:a2bn} 
Let $a,b$ be simple closed curves on a one-holed torus $S_{1,1}$ which intersect once transversely. 
Given integer $n \ge 3$ and length $L_{\partial}\ge 0$ of the neck geodesic of  $S_{1,1}$,  the geodesic length function $L_{a^2b^n}$ has exactly one minimum point $(L^*_{a}, L^*_{b}, L^*_{ab})\in {\rm Teich}^{\rm length}_{1,1}(L_{\partial})$, where $L^*_{b}$ is the unique positive solution to the equation
\begin{equation}\label{eqn:L^*b}
\frac{n}{2} \tanh\frac{nL^*_b}{4} \tanh\frac{L^*_b}{2}=1
\end{equation}
(hence $L^*_b$ is independent of $L_{\partial}$), and the minimum value $L^{\rm min}_{a^2b^n}$ is determined by
\begin{equation}\label{eqn:Lmin}
\sinh\frac{L^{\rm min}_{a^2b^n}}{4} =  \frac{n}{2}\cosh\frac{L_{\partial}}{4}\sinh\frac{nL^*_b}{4} \Big/ \cosh\frac{L^*_b}{2}. 
\end{equation}
\end{theorem}

\begin{proof}
Given a complete hyperbolic structure on $S_{1,1}$, we may choose a lifted holonomy representation 
$\rho: \pi_1(S_{1,1}) \rightarrow {\rm SL}(2,\mathbb R)$ so that
${\rm tr}\,\rho(a)>2$, ${\rm tr}\,\rho(b)>2$ and ${\rm tr}\,\rho(ab)>2$. 

We write $X=\rho(a)$ and $Y=\rho(b)$ so that the calculations done in \S \ref{s:calculation} can used directly with the notations therein.

By formulas (\ref{eqn:a2b2m}) and (\ref{eqn:a2b2m+1}) for ${\rm tr}(X^2Y^n)$ in \S \ref{s:calculation}, if $(L_{a},L_{b},L_{ab}) \in {\rm Teich}^{\rm length}_{1,1}(L_{\partial})$ is the minimum point of the length function $L_{a^2b^{n}}$, $n\ge 3$, there must hold that
${\rm tr}(XY^{m})={\rm tr}(XY^{m+1})$ if $n=2m+1$, or ${\rm tr}(XY^{m-1})={\rm tr}(XY^{m+1})$ if $n=2m$.

Thus we may assume $L_{ab^{m}}=L_{ab^{m+1}}$ if $n=2m+1$, and $L_{ab^{m-1}}=L_{ab^{m+1}}$ if $n=2m$, and proceed to minimizing $L_{a^2b^n}$.

First, we find a formula, (\ref{eqn:lengtha2bn-}) below, for $L_{a^2b^n}$.
We set
\begin{equation} 
Z:=XY^{\frac{n}{2}}.
\end{equation}
Then $X=ZY^{-\frac{n}{2}}$ and 
\begin{equation}\label{eqn:ZYZY}
X^2Y^n=ZY^{-\frac{n}{2}}ZY^{\frac{n}{2}}.
\end{equation}

It follows from the symmetry $L_{ab^{m}}=L_{ab^{m+1}}$ (if $n=2m+1$) or $L_{ab^{m-1}}=L_{ab^{m+1}}$ (if $n=2m$) that 
the axes of the fractional linear transformations $Y$ and $Z$ intersect perpendicularly.
Thus we may assume that
\begin{equation}\label{eqn:YandZ}
Y=\begin{pmatrix}e^{\frac{L_b}{2}} & 0 \\ 0 & e^{-\frac{L_b}{2}}\end{pmatrix}, \quad
Z=\begin{pmatrix}\cosh H_b & \sinh H_b \\ \sinh H_b & \cosh H_b\end{pmatrix},
\end{equation}
where $H_b$ is the length of the altitude perpendicular to the side of length $\frac12 L_b$ in the Weierstrass triangle associated to the triple $a, b, ab$.
By Proposition \ref{propn:sineplus}, we have
\begin{equation}\label{eqn:boundarycondition}
\sinh H_b \sinh\frac{L_b}{2}=\cosh\frac{L_{\partial}}{4}.
\end{equation}
Using (\ref{eqn:ZYZY}) and (\ref{eqn:YandZ}), a direct calculation gives 
\begin{equation} 
X^2Y^{n}=\begin{pmatrix} \cosh^2 H_b + (\sinh^2 H_b) e^{\frac{nL_b}{2}} & (\cosh H_b \sinh H_b) (e^{-\frac{nL_b}{2}}+1) \\ 
(\cosh H_b \sinh H_b) (1+e^{\frac{nL_b}{2}}) & (\sinh^2 H_b)e^{-\frac{nL_b}{2}} + \cosh^2 H_b \end{pmatrix}.
\end{equation}
Taking traces gives
\begin{equation}\label{eqn:tracea2bn}
2\cosh\frac{L_{a^2b^n}}{2}=2\cosh^2 H_b +2\sinh^2 H_b \cosh\frac{nL_b}{2}.
\end{equation}
Consequently, 
\begin{equation}\label{eqn:lengtha2bn}
\sinh\frac{L_{a^2b^n}}{4}=\sinh H_b \cosh\frac{nL_b}{4}.
\end{equation}
Combining with (\ref{eqn:boundarycondition}), we have
\begin{equation}\label{eqn:lengtha2bn-}
\sinh\frac{L_{a^2b^n}}{4}=\cosh\frac{L_{\partial}}{4}\cosh\frac{nL_b}{4}\Big / \sinh\frac{L_b}{2}.
\end{equation}

Next, we find the minimum point for $L_{a^2b^n}$.
Differentiating the two sides of (\ref{eqn:lengtha2bn-}) with respect to $L_b$, we see that 
\begin{equation}
\frac{{\rm d}L_{a^2b^n}}{{\rm d}L_b}=0 \Longleftrightarrow 
\frac{n}{4}\sinh\frac{nL_b}{4} \sinh\frac{L_b}{2} - \frac{1}{2}\cosh\frac{nL_b}{4}  \cosh\frac{L_b}{2} =0.
\end{equation}
Therefore at the minimum point $(L^*_{a}, L^*_{b}, L^*_{ab})$ for $L_{a^2b^n}$, the value of $L^*_b$ is determined by
\begin{equation}\label{eqn:Lb}
\frac{n}{2} \tanh\frac{L^*_b}{2} \tanh\frac{nL^*_b}{4} =1,
\end{equation}
which is equality (\ref{eqn:L^*b}) as desired.

It is easy to know that, the solution $L^*_b>0$ to equation (\ref{eqn:Lb}) exists and is unique, 
since the left side of (\ref{eqn:Lb}) is a strictly increasing function in $L^*_b \in (0,+\infty)$ which
has limit $0$ as $L^*_b \rightarrow 0$ and limit $n/2 >1$ as $L^*_b \rightarrow +\infty$.

We conclude that, given $n \ge 3$ and $L_{\partial}\ge 0$, the minimum $L^{\rm min}_{a^2b^n}$ of $L_{a^2b^n}$ is given by
\begin{equation}\label{eqn:minLa2bn-star}
\sinh\frac{L^{\rm min}_{a^2b^n}}{4}=\cosh\frac{L_{\partial}}{4}\cosh\frac{nL^*_b}{4}\Big / \sinh\frac{L^*_b}{2},
\end{equation}
where $L^*_b>0$ is determined by equation (\ref{eqn:Lb}). It follows from (\ref{eqn:Lb}) that
\begin{equation}
\cosh\frac{nL^*_b}{4}\Big / \sinh\frac{L^*_b}{2}=\sinh\frac{nL^*_b}{4} \Big/ \cosh\frac{L^*_b}{2}.
\end{equation}
Substituting this into (\ref{eqn:minLa2bn-star}) gives 
\begin{equation}\label{eqn:Lmin-2}
\sinh\frac{L^{\rm min}_{a^2b^n}}{4} =  \frac{n}{2}\cosh\frac{L_{\partial}}{4}\sinh\frac{nL^*_b}{4} \Big/ \cosh\frac{L^*_b}{2},
\end{equation}
which is equality (\ref{eqn:Lmin}) as desired. This finishes the proof of Theorem \ref{thm:a2bn}.
\end{proof}

The solution to (\ref{eqn:Lb}) defines a function 
\begin{equation}
L^*_b=L^*_b(L_{\partial}, n), 
\end{equation}
and consequently a function
\begin{equation}
L^{\rm min}_{a^2b^n}=L^{\rm min}_{a^2b^n}(L_{\partial}, n). 
\end{equation}

In the following addendum we obtain monotonicity and asymptotic properties of $L^*_b$ and $L^{\rm min}_{a^2b^n}$ 
as one of $n$ and $L_{\partial}$ tends to infinity while the other remains fixed.

\begin{theorem}\label{rmk:a2bn} We have the following addenda to Theorem \ref{thm:a2bn}.

(a) For $L_{\partial}\ge 0$ fixed, both $L^*_b$ and $(nL^*_b)/4$, where $n \ge 3$ are strictly decreasing in $n$. 
Furthermore, for $L_{\partial}\ge 0$ fixed and $n \rightarrow +\infty$, there hold 
\begin{equation} 
L^*_b \rightarrow 0,
\end{equation} 
\begin{equation} 
\frac{nL^*_b}{4} \rightarrow t^*, 
\end{equation}
\begin{equation} \label{eqn:Lmin-sympt-in-n}
L^{\rm min}_{a^2b^n} \sim 4 \log n,
\end{equation}
where  $t^*=1.199678640...$ is the unique positive solution to the equation 
\begin{equation} \label{eqn:t^*}
t^* \tanh(t^*)=1.
\end{equation}

(b) As $n\ge 3$ fixed, $L^{\rm min}_{a^2b^n}$ is strictly increasing in $L_{\partial} \in [0,+\infty)$. Furthermore, as $n\ge 3$ fixed and
$L_{\partial} \rightarrow +\infty$, there hold 
\begin{equation}
L^{\rm min}_{a^2b^n} \rightarrow +\infty,
\end{equation}
\begin{equation} \label{eqn:Lmin-sympt-in-Lpartial}
L^{\rm min}_{a^2b^n} \sim  L_{\partial}. 
\end{equation}
\end{theorem}

\begin{proof}
For (a), let $L_{\partial}\ge 0$ be fixed. 
Note that the left side of (\ref{eqn:Lb}) is an expression which is strictly increasing in $n$ and in $L^*_b$.
It follows that $L^*_b$ is strictly decreasing in $n$. For the limit as $n\rightarrow +\infty$, there must hold
\begin{equation}
\lim_{n\rightarrow +\infty}L^*_b=0,
\end{equation}
otherwise there would be a contradiction by taking limit in equality (\ref{eqn:Lb}).

Rewriting equality (\ref{eqn:Lb}) as 
\begin{equation}\label{eqn:Lb-rewritten}
\frac{nL^*_b}{4} \tanh\frac{nL^*_b}{4} =\frac{L^*_b}{2} \left( \tanh\frac{L^*_b}{2} \right)^{-1}, 
\end{equation}
one concludes that $\frac{nL^*_b}{4}$ is strictly decreasing in $n$.  

Let us write
\begin{equation}
t^*=\lim_{n\rightarrow +\infty} \frac{nL^*_b}{4}.
\end{equation}
Then $t^*>0$. Taking limits in (\ref{eqn:Lb-rewritten}) as $n \rightarrow +\infty$ gives 
\begin{equation*}
t^*\tanh t^* =1,
\end{equation*}
which is (\ref{eqn:t^*}) as desired. 
Numerically solving this equation using Maple$^\circledR$ gives 
\begin{equation}
t^*=1.199678640257733833916369848641141944261458788418607...
\end{equation}

It is easy to derive from (\ref{eqn:Lmin}) that, as $L_{\partial} \ge 0$ fixed and $n\rightarrow +\infty$, 
\begin{equation*}
\exp \Big( \frac{1}{4}L^{\rm min}_{a^2b^n} \Big)  \sim n \cosh\frac{L_{\partial}}{4}\sinh t^*,
\end{equation*}
and hence
\begin{equation*}
L^{\rm min}_{a^2b^n}  \sim 4\log n, 
\end{equation*}
which is (\ref{eqn:Lmin-sympt-in-n}) as desired.

For (b), let $n \ge 3$ be fixed. Then $L^*_b$ is fixed since it depends on $n$ only. 
Thus it is easy to see from (\ref{eqn:Lmin}) that $L^{\rm min}_{a^2b^n}$ is strictly increasing in $L_{\partial} \in [0,+\infty)$ and
\begin{equation*}
\lim_{L_{\partial} \rightarrow+\infty} L^{\rm min}_{a^2b^n}  = +\infty.
\end{equation*}
Furthermore, one has, as $n \ge 3$ fixed and $L_{\partial} \rightarrow+\infty$, 
\begin{equation*}
\exp \Big( \frac{1}{4}L^{\rm min}_{a^2b^n} \Big)  \sim 
\exp \Big(\frac{L_{\partial}}{4} \Big)\,\frac{n}{2} \sinh \frac{nL^*_b}{4} \Big/ \cosh \frac{L^*_b}{2},
\end{equation*}
and hence
\begin{equation*}
L^{\rm min}_{a^2b^n}  \sim L_{\partial},
\end{equation*}
which is (\ref{eqn:Lmin-sympt-in-Lpartial}) as desired. This finishes the proof of Theorem \ref{rmk:a2bn}.
\end{proof}


\vskip 12pt

{\bf Acknowledgements.} 
Part of the work contained in this paper was done during the authors' visit in the summer of 2021 to the Institute for Advanced Study in Mathematics (IASM) at Zhejiang University, and they thank IASM for its hospitality, and Professor Shicheng Wang for his constant interests and helpful discussions. They would like to thank the anonymous referee for his/her helpful comments that improved the quality of the manuscript. The second author is supported by NSFC grant No. 12171345.


\end{document}